\documentclass[12pt,reqno]{amsart}
\newtheorem{thm}{Theorem}[section]
\newtheorem{defi}{Definition}[section]
\newtheorem{lem}{Lemma}[section]
\newtheorem{cor}{Corollary}[section]
\newtheorem{prop}{Proposition}[section]
\theoremstyle{definition}
\newtheorem{rem}{Remark}[section]
\newtheorem{exmp}{Example}[section]
\newcommand{\be}{\begin{equation}}
\newcommand{\ee}{\end{equation}}
\newcommand{\bea}{\begin{eqnarray}}
\newcommand{\eea}{\end{eqnarray}}
\newcommand{\beb}{\begin{eqnarray*}}
\newcommand{\eeb}{\end{eqnarray*}}

\usepackage{amssymb,amsfonts,amsthm,setspace,indentfirst,mathrsfs}
\usepackage{minibox}
\usepackage{enumitem}
\usepackage{mathtools}
\usepackage{hyperref}
\numberwithin{equation}{section}
\begin{document}

\title[A study on $\mathcal I$-localized sequences in S-metric spaces]{  A study on $\mathcal I$-localized sequences in S-metric spaces}

\author[A.K.Banerjee]{Amar Kumar Banerjee$^1$ }
\author[N.Hossain]{Nesar Hossain$^2$}

\address{$^{1,2,}$Department of Mathematics, The University of Burdwan, Golapbag, Burdwan - 713104, West Bengal, India.}

\email{$^1$akbanerjee1971@gmail.com, akbanerjee@math.buruniv.ac.in}
\email{$^2$nesarhossain24@gmail.com}

\subjclass[2020]{40A35, 40A30, 54A20, 40A05}
\keywords{Ideal,  $S$-metric space, locator, $\mathcal{I}$-localized sequence, $\mathcal{I^*}$-localized sequence, $\mathcal{I}$-barrier, uniformly $\mathcal{I}$-localized sequence. }

\maketitle

\begin{abstract}
In this paper we study the notion of $\mathcal{I}$-localized and $\mathcal{I^*}$-localized sequences in $S$-metric spaces. Also, we investigate some properties related to $\mathcal{I}$-localized and $\mathcal{I}$-Cauchy sequences and give the idea of  $\mathcal{I}$-barrier of a sequence in the same space.
\end{abstract}

\section{\bf{Introduction }}
After long $50$ years of introducing of the notion of statistical convergence \cite{Fast, Steinhaus, Schoenberg} the idea of $\mathcal{I}$-convergence were given by Kostyrko et al. \cite{Kostyrko Salat Wilczynski} in $2000$ where $\mathcal{I}$ is an ideal of subsets of the set of natural numbers. Then this ideas of ideal convergence were studied by several authors in many directions \cite{Banerjee, Banerjee Banerjee 2018, Banerjee Mondal, Das Kostyrko Wilczynski Malik, Lahiri Das}.

 The notion of localized sequence was introduced by Krivonosov \cite{Krivonosov} in metric spaces in $1974$ as a generalization of a Cauchy sequence. A sequence $\{x_n\}_{n\in\mathbb{N}}$ of points in a metric space $(X,d)$ is said to be localized in some subset $M\subset X$ if the number sequence $\alpha_n=d(x_n , x)$ converges for $x\in M$. The maximal subset of $X$ on which the sequence $\{x_n\}_{n\in\mathbb{N}}$  is localized, is called the locator of $\{x_n\}_{n\in\mathbb{N}}$ and it is denoted by $loc(x_n)$. If $\{x_n\}_{n\in\mathbb{N}}$ is localized on $X$ then it is called localized everywhere in $X$. If the locator of a sequence $\{x_n\}_{n\in\mathbb{N}}$ contains all elements of this sequence, except for a finite number of elements of it then the sequence $\{x_n\}_{n\in\mathbb{N}}$ is called localized in itself. After long years, in $2020$, Nabiev et el. \cite{Nabiev Savas Gurdal}  introduced the idea of $\mathcal{I}$-localized and $\mathcal{I^*}$-localized sequences in metric spaces and investigated some basic properties of the $\mathcal{I}$-localized sequences related with $\mathcal{I}$-Cauchy sequences.  At the same time Gürdal et el. \cite{Gurdal Sari Savas} studied $A$-statistically localized sequences in $n$-normed spaces, Yamanci et el. \cite{Yamanci Nabiev Gurdal} have extended this idea of localized sequences to statistically localized sequences in $2$-normed spaces and interestingly this  notion has been generalized in ideal context in $2$-normed spaces by Yamanci et el. \cite{Yamanci Savas Gurdal}. In $2021$, Granados and Bermudez \cite{Granados Bermudez}  studied on $\mathcal{I}_2$-localized double sequences and Granados \cite{Granados}  nurtured this notion with the help of triple sequences using ideals  in metric spaces.

In $2012$  Sedghi et al. \cite{Sedghi} has defined the interesting notion of $S$-metric spaces and proved some basic properties in this space. For an admissible ideal $\mathcal{I}$,  $\mathcal{I^*}$-convergence and $\mathcal{I^*}$-Cauchy criteria in $X$ imply $\mathcal{I}$-convergence and 
$\mathcal{I}$-Cauchy criteria in $X$ respectively. Moreover, for admissible ideal with the property $(AP)$, $\mathcal{I}$ and $\mathcal{I^*}$-convergence ($\mathcal{I}$ and $\mathcal{I^*}$-Cauchy criteria) in $X$ are equivalent \cite{Banerjee Banerjee 2018}.  In this paper we have studied the notion of $\mathcal{I}$ and $\mathcal{I^*}$-localized sequences and investigated some results related to $\mathcal{I}$-Cauchy sequences in $S$-metric spaces.

\section{\textbf{Preliminaries}}
Now we recall some basic definitions and notations.  If $X$ is a non-empty set then a collection $\mathcal{I}$ of subsets of $X$ is said to be an ideal of $X$ if $(i)$ $A,B\in \mathcal{I} \Rightarrow A\cup B\in \mathcal{I}$ and $(ii)$ $A\in \mathcal{I}, B\subset A \Rightarrow B\in \mathcal{I}$. Clearly $\{\phi\}$ and $2^X$, the power set of $X$, are the trivial ideal of $X$. A non trivial ideal $\mathcal{I}$ is said to be an admissible ideal if $\{x\}\in \mathcal{I}$ for each $x\in X$.  If $\mathcal{I}$ is a non trivial ideal of $X$ then the family of sets $\mathcal{F(I)}= \{A\subset X: X\setminus A\in \mathcal{I}\}$ is clearly a filter on $X$. This filter is called the filter associated with the ideal $\mathcal{I}$. An admissible ideal $\mathcal{I}$ of $\mathbb{N}$, the set of natural numbers, is said to satisfy the condition $(AP)$ if for every countable family $\{A_1,A_2,A_3,\ldots\}$ of sets belonging to $\mathcal{I}$ there exists a countable family of sets $B_1,B_2,B_3,\ldots\}$ such that $A_i\triangle B_i$ is a finite set for each $i\in \mathbb{N}$ and $B=\bigcup_{i\in\mathbb{N}}B_i \in \mathcal{I}$.
Note that $B_i\in \mathcal{I}$ for all $i\in \mathbb{N}$.

Now we  recall some basic definitions and some  properties from (\cite{Sedghi}).
\begin{defi}
Let $X$ be  a non empty set. An $S$-metric on $X$ is a function $S:X\times X\times X\rightarrow [0,\infty)$, such that for each $x,y,z,a\in X$, \\
$(i)$ $S(x,y,z)\geq 0 $;\\
$(ii)$ $S(x,y,z)=0$ if and only if $x=y=z$;\\
$(iii)$ $S(x,y,z)\leq S(x,x,a)+S(y,y,a)+S(z,z,a)$.
\end{defi}
The pair $(X,S)$ is called an $S$-metric space. Several examples may be seen from \cite{Sedghi}. In an $S$-metric space, we have $S(x,x,y)=S(y,y,x)$. A sequence $\{x_n\}_{n\in \mathbb{N}}$ in $(X,S)$ is said to converge to $x$ if and only if $S(x_n,x_n,x)\rightarrow 0$ as $n\rightarrow \infty$. That is for $\varepsilon >0$ there exists $n_0\in \mathbb{N}$ such that for all $n\geq n_0$, $S(x_n,x_n,x)<\varepsilon$. The sequence $\{x_n\}_{n\in \mathbb{N}}$ in $(X,S)$ is called a Cauchy sequence if for each $\varepsilon >0$, there exists $n_0\in \mathbb{N}$ such that $S(x_n,x_n,x_m)<\varepsilon$ for each $n,m\geq n_0$. 

 We recall the following definitions in an $S$-metric space from \cite{Banerjee Banerjee 2018} which will be useful in the sequal.

A sequence $\{x_n\}_{n\in \mathbb{N}}$ of elements of $X$ is said to be $\mathcal{I}$-convergent to $x\in X$ if for each $\varepsilon >0$, the set $A(\varepsilon)= \{n\in \mathbb{N} : S(x_n,x_n,x)\geq \varepsilon \}\in \mathcal{I}$. The  sequence $\{x_n\}_{n\in \mathbb{N}}$ of elements of $X$ is said to be $\mathcal{I^*}$-convergent to $x\in X$ if and only if there exists a set $M\in \mathcal{F(I)}$, $M=\{m_1<m_2<\cdots <m_k<\cdots\}\subset \mathbb{N}$ such that $\lim _{k\to \infty}S(x_{m_k},x_{m_k},x)=0$.

A  sequence $\{x_n\}_{n\in \mathbb{N}}$ of elements of $X$ is called an $\mathcal{I}$-Cauchy sequence if for every $\varepsilon >0$ there exists a positive integer $n_0 = n_0 (\varepsilon)$ such that the set $A(\varepsilon)= \{n\in \mathbb{N}: S(x_n,x_n, x_{n_0})\geq \varepsilon \}\in \mathcal{I}$.
The sequence $\{x_n\}_{n\in \mathbb{N}}$ of elements of $X$ is called an $\mathcal{I^*}$-Cauchy sequence if there exists a set $M=\{m_1<m_2<\ldots<m_k \ldots \}\subset \mathbb{N}, M\in \mathcal{F(I)}$ such that the subsequence $\{x_{m_k}\}$ is an ordinary Cauchy sequence in $X$ i.e.,for each preassigned $\varepsilon >0$ there exists $k_0\in \mathbb{N}$ such that $S(x_{m_k},x_{m_k},x_{m_r})<\varepsilon$ for all $k,r\geq k_0$.


\section{\bf{Main Results}}
Throughout the discussion $\mathbb{N}$ stands for the set of natural numbers, $\mathcal{I}$  for an admissible ideal of $\mathbb{N}$ and  $X$  stands for an $S$-metric space unless other stated. Now we introduce some definitions and properties regarding localized sequences with respect to the ideal $\mathcal{I}$ in $S$-metric spaces. 
\begin{defi}
A sequence $\{x_n\}_{n\in\mathbb{N}}$ in $X$ is said to be localized in the subset $M\subset X$ if for each $x\in M$, the non negative real sequence $\{S(x_n,x_n,x)\}_{n\in \mathbb{N}}$ converges in $X$.
\end{defi}
\begin{defi}\label{I-localized}
A sequence $\{x_n\}_{n\in \mathbb{N}}$ of elements of $X$ is said to be $\mathcal{I}$-localized in the subset $M\subset X$ if for each $x\in M$, $\mathcal{I}\text{-}\lim_{n\to \infty} S(x_n,x_n,x)$ exists i.e. if the non negative real sequence $ \{S(x_n,x_n,x)\}_{n\in\mathbb{N}}$ is $\mathcal{I}$-convergent.
\end{defi}

The maximal subset of $X$  on which a sequence $\{x_n\}_{n\in \mathbb{N}}$ in $X$ is $\mathcal{I}$-localized, is called the $\mathcal{I}$-locator of  $\{x_n\}_{n\in \mathbb{N}}$ and it is  denoted by $loc _{\mathcal{I}} (x_n)$.
A sequence $\{x_n\}_{n\in \mathbb{N}}$ in $X$ is said to be $\mathcal{I}$-localized everywhere if the  $\mathcal{I}$-locator of  $\{x_n\}_{n\in \mathbb{N}}$ is the whole set $X$. 
The sequence $\{x_n\}_{n\in \mathbb{N}}$ is said to be $\mathcal{I}$-localized in itself if the set $\{n \in \mathbb{N}: x_n \in loc _{\mathcal{I}} (x_n)\}\in \mathcal{F(I)}$.

Now we intoduce an important result in $S$-metric spaces which will be useful in the sequal.
\begin{lem}\label{lem0.1}
The inequality $|S(x,x,\xi)-S(\xi,\xi,y)|\leq 2S(x,x,y)$ holds good for any $x,y,\xi\in X$.
\end{lem}

\begin{proof}
Now for $x,y,\xi\in X$, we have \begin{align*}
    S(x,x,\xi)
&\leq S(x,x,y)+S(x,x,y)+S(\xi,\xi,y)\\
&=2S(x,x,y)+S(\xi,\xi,y)
\end{align*}
Therefore \begin{equation}\label{eqn3.1}
    S(x,x,\xi)-S(\xi,\xi,y)\leq 2S(x,x,y)
\end{equation}
Again we have, \begin{align*}
    &S(\xi,\xi,y)-S(x,x,\xi)\\ 
    = &S(y,y,\xi)-S(x,x,\xi)\\
\leq  &S(y,y,x)+S(y,y,x)+S(\xi,\xi,x)-S(x,x,\xi)\\
= &S(x,x,y)+S(x,x,y)+S(x,x,\xi)-S(x,x,\xi)\\
= &2S(x,x,y)
\end{align*}
Therefore \begin{equation}\label{eqn3.2}
  S(\xi,\xi,y)-S(x,x,\xi)\leq 2S(x,x,y) 
\end{equation}
From the equations  (\ref{eqn3.1}) and (\ref{eqn3.2}) we have $|S(x,x,\xi)-S(\xi,\xi,y)|\leq 2S(x,x,y)$. This completes the proof.
\end{proof}

\begin{lem}\label{lem 3.1}
 If $\{x_n\}_{n\in \mathbb{N}}$ is an $\mathcal{I}$-Cauchy sequence in $X$ then it is $\mathcal{I}$-localized everywhere. 
 \end{lem}
 \begin{proof}
 By the condition, for every $\varepsilon>0$ there exists a positive integer $n_0=n_0(\varepsilon)$ such that   the set $A(\varepsilon)= \{n\in \mathbb{N}: S(x_n,x_n, x_{n_0})\geq \frac{\varepsilon}{2} \}\in \mathcal{I}$. Let $\xi \in X$. Using the Lemma \ref{lem0.1}, we have
  $|S(x_n,x_n,\xi)-S(\xi , \xi , x_{n_0})|\leq 2S(x_n , x_n , x_{n_0})$. Therefore $\{n\in \mathbb{N}:|S(x_n,x_n,\xi)-S(\xi , \xi , x_{n_0})|\geq \varepsilon\}\subset \{n\in \mathbb{N}: S(x_n , x_n , x_{n_0}) \geq \frac{\varepsilon}{2}\}\in \mathcal{I}$. This shows that the number sequence $\{S(x_n,x_n,\xi)\}_{n\in\mathbb{N}}$ is $\mathcal{I}$-convergent for each $\xi\in X$. Hence the sequence $\{x_n\}_{n\in \mathbb{N}}$ is $\mathcal{I}$-localized everywhere.
\end{proof}

\begin{cor}
By the Lemma \ref{lem 3.1} it follows that 
 every $\mathcal{I}$-convergent sequence in $X$ is $\mathcal{I}$-localized everywhere.
\end{cor}
Also,  if $\mathcal{I}$ is an admissible ideal, then every localized sequence in $X$ is $\mathcal{I}$-localized sequence in 
$X$.

\begin{defi}\label{defi2}
A sequence $\{x_n\}_{n\in \mathbb{N}}$ is said to be $\mathcal{I^*}$-localized in $X$ if the real sequence $\{S(x_n,x_n,x)\}_{n\in\mathbb{N}}$ is $\mathcal{I^*}$-convergent for each $x\in X$.
\end{defi}

\begin{thm}\label{thm3.1}
Let $\mathcal{I}$ be an admissible ideal. If a sequence $\{x_n\}_{n\in \mathbb{N}}$ in $X$ is $\mathcal{I^*}$-localized on the subset $M\subset X$, then $\{x_n\}_{n\in \mathbb{N}}$ is $\mathcal{I}$-localized on the set $M$ and $loc _{\mathcal{I^*}} (x_n)\subset loc _{\mathcal{I}} (x_n) $.
\end{thm}

\begin{proof}
Let $\{x_n\}_{n\in \mathbb{N}}$ be $\mathcal{I^*}$-localized on the subset $M\subset X$. Then, by the definition \ref{defi2},  the number sequence $\{S(x_n,x_n,x)\}_{n\in\mathbb{N}}$ is $\mathcal{I^*}$-convergent for each $x\in M$. Now since $\mathcal{I}$ is an admissible ideal,  the number sequence $\{S(x_n,x_n,x)\}_{n\in\mathbb{N}}$ is $\mathcal{I}$-convergent for each $x\in M$ which gives  $\{x_n\}_{n\in \mathbb{N}}$ is $\mathcal{I}$-localized on the set $M$.
\end{proof}
But the converse of the Theorem \ref{thm3.1} does not hold in general. It can be shown by the following example.

\begin{exmp}
First we define the $S$-metric on $\mathbb{R}$ by $S(x,y,z)=d(x,z)+d(y,z)$, $\forall \ x,y,z \in \mathbb{R}$ where $d$ is the usual metric on $\mathbb{R}$. Let $\mathbb{N}= \bigcup _{j=1} ^{\infty} \triangle _j$ be a decomposition of $\mathbb{N}$ such that each $\triangle_j$ is infinite and obviously $\triangle_i\cap \triangle_j=\phi$ for $i\neq j$. Let $\mathcal{I}$ be the class of all those subsets of $ \mathbb{N}$ which intersects only a finite number of $\triangle_j 's$. Then $\mathcal{I}$ is an admissible ideal on $\mathbb{N}$. Let $\{x_n\}_{n\in\mathbb{N}}$ be a sequence in $(\mathbb{R},S)$ defined by $x_n=\frac{1}{j}$, for $n\in \triangle_j$. Let $\varepsilon>0$ be given. Now since the sequence $\{\frac{1}{j}\}_{j\in\mathbb{N}}$ in $(\mathbb{R},d)$ converges to zero, so there exists $p\in \mathbb{N}$ such that $d(\frac{1}{j},0)<\frac{\varepsilon}{4}$ for all $j\geq p$. Now \begin{equation}\label{eqn1}
S(x_n,x_n,0)=d(x_n,0)+d(x_n,0)=d(\frac{1}{j},0)+d(\frac{1}{j},0)<\frac{\varepsilon}{4}+\frac{\varepsilon}{4}=\frac{\varepsilon}{2},\ \text{for all}\ j\geq p.
\end{equation} Let $x\in \mathbb{R}$. Now using the Lemma \ref{lem0.1} and the equation (\ref{eqn1}), we have \begin{equation*}
   |S(x_n,x_n,x)-S(x,x,0)|\leq 2S(x_n,x_n,0)<\varepsilon,\ \text{ for all}\ j\geq p. 
\end{equation*} Hence $\{n\in\mathbb{N}: |S(x_n,x_n,x)-S(x,x,0)|\geq \varepsilon\}\subseteq \triangle_1\cup\triangle_2\cup\cdots\cup\triangle_p\in\mathcal{I}$. Therefore $\{n\in\mathbb{N}: |S(x_n,x_n,x)-S(x,x,0)|\geq \varepsilon\}\in\mathcal{I}$. Hence for each $x\in \mathbb{R}$, the number sequence $\{S(x_n,x_n,x)\}_{n\in\mathbb{N}}$ is $\mathcal{I}$-convergent. Therefore  the sequence $\{x_n\}_{n\in\mathbb{N}}$ is $\mathcal{I}$-localized in $(\mathbb{R},S)$ 

Now we show that the sequence $\{x_n\}_{n\in\mathbb{N}}$ is not  $\mathcal{I}^*$-localized in $(\mathbb{R},S)$. If possible, let 
 the sequence $\{x_n\}_{n\in\mathbb{N}}$ be $\mathcal{I}^*$-localized in $(\mathbb{R},S)$. So the number sequence $\{S(x_n,x_n,x)\}_{n\in\mathbb{N}}$ is $\mathcal{I}^*$-convergent for each $x\in \mathbb{R}$. So there exists $A\in \mathcal{I}$ such that, for $M=\mathbb{N}\setminus A=\{m_1<m_2<\cdots<m_k<\cdots\}\in\mathcal{F(I)}$, the subsequence $\{S(x_n,x_n,x)\}_{n\in M}$ is convergent. Now by the definition of $\mathcal{I}$, there is a positive integer $t$ such that $A\subseteq \triangle_1\cup\triangle_2\cup\ldots\cup\triangle_t$. But then $\triangle_{i}\subset \mathbb{N}\setminus A=M$ for all $i\geq t+1$. In particular $\triangle_{t+1}$, $\triangle_{t+2}\subset M$. Since $\triangle_j's$ are infinite, there are infinitely many $k's$ for which $x_{m_k}=\frac{1}{t+1}$ when $m_k\in \triangle_{t+1}$ and $x_{m_k}=\frac{1}{t+2}$ when $m_k\in \triangle_{t+2}$. So \begin{equation*}
     S(x_{m_k},x_{m_k},0)=\begin{cases}
     d(\frac{1}{t+1},0)+d(\frac{1}{t+1},0)=\frac{2}{t+1} \ \text{when}\ m_k\in \triangle_{t+1}\\
     d(\frac{1}{t+2},0)+d(\frac{1}{t+2},0)=\frac{2}{t+2} \ \text{when}\ m_k\in\triangle_{t+2}
     \end{cases}.
 \end{equation*}
 So for $0\in\mathbb{R}$ there are infinitely many terms of the form $\frac{2}{t+1}$ and $\frac{2}{t+2}$. So $\{S(x_{m_k},x_{m_k},0)\}$ can not be convergent which leads to a contradiction. Hence the sequence $\{x_n\}_{n\in\mathbb{N}}$ can not be $\mathcal{I^*}$-localized.
\end{exmp}
\begin{rem}
 If $X$ has no limit point, then it follows that $\mathcal{I}$-convergence and $\mathcal{I^*}$-convergence coincide. Therefore by the Definitions \ref{I-localized} and \ref{defi2} and by the Theorem \ref{thm3.1} we have $loc _{\mathcal{I}} (x_n)=loc _{\mathcal{I^*}} (x_n)$. 
Also, if $X$ has a limit point $\xi$ then there is an admissible ideal $\mathcal{I}$ for which there exists an $\mathcal{I}$-localized sequence $\{y_n\}_{n\in \mathbb{N}}$ in $X$ but $\{y_n\}_{n\in \mathbb{N}}$ is not $\mathcal{I^*}$-localized.
\end{rem}

Now we shall formulate the necessary and sufficient condition for the ideal $\mathcal{I}$ under which $\mathcal{I}$ and $\mathcal{I^*}$-localized sequences are equivalent.

\begin{thm}
$(i)$ If $\mathcal{I}$ satisfies the condition $(AP)$ and $\{x_n\}_{n\in \mathbb{N}}$ is an $\mathcal{I}$-localized on the set $M\subset X$ then it is $\mathcal{I^*}$-localized on $M$.\\
$(ii)$ If $X$ has a limit point and every $\mathcal{I}$-localized sequence implies $\mathcal{I^*}$-localized then $\mathcal{I}$ will have the property $(AP)$. 
\end{thm}

\begin{proof}
$(i)$ Suppose that $\mathcal{I}$ satisfies the condition $(AP)$ and $\{x_n\}_{n\in \mathbb{N}}$ is an $\mathcal{I}$-localized on the set $L\subset X$. Then, by the definition, the number sequence $\{S(x_n,x_n,x)\}_{n\in\mathbb{N}}$ is $\mathcal{I}$-convergent for $x\in L$. Let $\{S(x_n,x_n,x)\}_{n\in\mathbb{N}}$ be $\mathcal{I}$-convergent to $\beta=\beta(x) \in \mathbb{R}^+$. Then for each $\varepsilon >0$ the set $A(\varepsilon)= \{n\in \mathbb{N}: |S(x_n,x_n,x)-\beta|\geq \varepsilon \}\in \mathcal{I}$. Now suppose $A_1= \{n\in \mathbb{N}: |S(x_n,x_n,x)-\beta|\geq 1 \}$ and $A_k= \{n\in \mathbb{N}: \frac{1}{k}\leq|S(x_n,x_n,x)-\beta|< \frac{1}{k-1} \}$ for $k\geq 2, k\in \mathbb{N}$. Obviously $A_i \cap A_j=\phi$, for $i\neq j$. Since $\mathcal{I}$ satisfies the condition $(AP)$, there exists a countable family of sets $\{B_1,B_2,\cdots\}$ such that $A_j\triangle B_j$ is finite for $j\in \mathbb{N}$ and $B=\bigcup_{j=1}^{\infty}B_j \in \mathcal{I}$. Now we shall show that the sequence $\{x_n\}_{n\in \mathbb{N}}$ is $\mathcal{I^*}$-localized. By the definition, it is enough to prove that the number sequence $\{S(x_n,x_n,x)\}_{n\in\mathbb{N}}$ is $\mathcal{I^*}$-convergent for every $x\in L$. We show  for $\mathbb{N}\setminus B=M=\{m_1<m_2<\cdots<m_k<\cdots\}\in \mathcal{F(I)}$, $\lim_{n\to \infty, n\in M}S(x_n,x_n,x)=\beta$. Let $\theta>0$ and $k\in \mathbb{N}$ be such that $\frac{1}{k+1}<\theta$. Then $\{n\in \mathbb{N}: |S(x_n,x_n,x)-\beta|\geq \theta\}\subset \bigcup_{j=1}^{k+1}A_j$. Since $A_j\triangle B_j$, $j=1,2,\cdots k+1$, is finite, we have an $n_0\in \mathbb{N}$ such that $(\bigcup_{j=1}^{k+1}B_j)\cap \{n\in \mathbb{N}:n>n_0\}=(\bigcup_{j=1}^{k+1}A_j)\cap \{n\in \mathbb{N}: n>n_0\}$. If $n>n_0$ and $n\notin B$, then $n\notin \bigcup_{j=1}^{k+1}B_j$ and so  $n\notin \bigcup_{j=1}^{k+1}A_j$. But then $|S(x_n,x_n,x)-\beta|<\frac{1}{k+1}<\theta$. Thus the number sequence $\{S(x_n,x_n,x)\}_{n\in\mathbb{N}}$, $x\in L$, is $\mathcal{I^*}$-convergent. Therefore the sequence $\{x_n\}_{n\in \mathbb{N}}$ is $\mathcal{I^*}$-localized.\\
$(ii)$ The proof  is parallel to the Theorem 3.2 of \cite{Kostyrko Salat Wilczynski}). So it is omitted.
\end{proof}

\begin{defi}
Let $\{x_n\}_{n\in\mathbb{N}}$ be a sequence in $X$. Then $\{x_n\}_{n\in\mathbb{N}}$ is said to be $\mathcal{I}$-bounded  if there exists $x\in X$ and $G>0$ such that the set $\{n\in\mathbb{N}:S(x_n,x_n,x)>G\}\in\mathcal{I}$.
\end{defi}

\begin{prop}
Every $\mathcal{I}$-localized sequence is $\mathcal{I}$-bounded.
\end{prop}

\begin{proof}
Let $\{x_n\}_{n\in\mathbb{N}}$ be  $\mathcal{I}$-localized on a subset $M\subset X$. Then the number sequence $ \{S(x_n,x_n,\xi)\}_{n\in\mathbb{N}}$ is $\mathcal{I}$-convergent for every $\xi \in M$. Let $ \{S(x_n,x_n,\xi)\}_{n\in\mathbb{N}}$ converges to $\alpha=\alpha(\xi)\in \mathbb{R}$. Let $G>0$ be given. Then  $\{n\in \mathbb{N}: |S(x_n,x_n,\xi)-\alpha| >G \}\in \mathcal{I} $.  This implies that $\{n\in \mathbb{N}: S(x_n,x_n,\xi)-\alpha >G \}\cup \{n\in \mathbb{N}: S(x_n,x_n,\xi)-\alpha <-G \}\in\mathcal{I}$. Therefore $\{n\in \mathbb{N}: S(x_n,x_n,\xi)>\alpha+G \}\in\mathcal{I}$, which shows that the sequence $\{x_n\}_{n\in\mathbb{N}}$ is $\mathcal{I}$-bounded.
\end{proof}

\begin{thm}\label{thm4}
Let $\mathcal{I}$ be an admissible ideal with the condition $(AP)$ and $L=loc_{\mathcal{I}}(x_n)$ and let  $z\in X$ be a point such that for any $\varepsilon >0$ there exists $x\in L$ satisfying
\begin{equation}\label{equation 3.1 of theorem 3.1}
   \{n\in \mathbb{N}: |S(x_n,x_n,x)-S(x_n,x_n,z)|\geq \varepsilon \}\in \mathcal{I}. 
\end{equation}
Then $z\in L$.
\end{thm}

\begin{proof}
 Let $\varepsilon >0$ be given and $x\in L=loc_{\mathcal{I}}(x_n)$ be a point satisfying the condition (\ref{equation 3.1 of theorem 3.1}).
Let $A=\{n\in \mathbb{N}: |S(x_n,x_n,x)-S(x_n,x_n,z)|\geq \varepsilon \}\in \mathcal{I}$. Then $M=\mathbb{N}\setminus A=\{k_1<k_2<\cdots<k_n<\cdots\}\in \mathcal{F(I)}$.Therefore, for $k_n\in M$, we have $|S(x_n,x_n,x)-S(x_n,x_n,z)|< \varepsilon$. Now since  $x\in L=loc_{\mathcal{I}}(x_n)$, the number sequence $\{S(x_n,x_n,x)\}_{n\in\mathbb{N}}$ is $\mathcal{I}$-convergent. So the number sequence $\{S(x_n,x_n,x)\}_{n\in\mathbb{N}}$ is $\mathcal{I}$-Cauchy. Again since $\mathcal{I}$ satisfies the condition $(AP)$, the number sequence $\{S(x_n,x_n,x)\}_{n\in\mathbb{N}}$ is $\mathcal{I^*}$-Cauchy. Then there exists $B=\{m_1<m_2<\cdots<m_k<\cdots\}\in\mathcal{F(I)}$ such that the subsequence $\{S(x_{m_k},x_{m_k},x)\}$ is a Cauchy sequence i.e., there exists $n_0\in\mathbb{N}$ such that $|S(x_{m_r},x_{m_r},x)-S(x_{m_k},x_{m_k},x)|<\varepsilon$ for all $r,k>n_0$. Since $M\cap B\in\mathcal{F(I)}$, let us enumerate the set $M\cap B=K=\{n_1<n_2<\cdots<n_p<\cdots\}\in\mathcal{F(I)}$. Then for  $n_p,n_q\in K$ and $n_p,n_q>m_{n_0}$, we have \begin{align*}
    &|S(x_{n_p},x_{n_p},z)-S(x_{n_q},x_{n_q},z)|\\
    \leq & |S(x_{n_p},x_{n_p},z)-S(x_{n_p},x_{n_p},x)|+|S(x_{n_p},x_{n_p},x)-S((x_{n_q},x_{n_q},x)| \\ &+|S((x_{n_q},x_{n_q},x)-S((x_{n_q},x_{n_q},z)|\\
    < & \varepsilon+\varepsilon+\varepsilon \\
    = & 3\varepsilon. 
\end{align*}
 Therefore we have the subsequence $\{S(x_n,x_n,z)\}_{n\in K}$ is a Cauchy Sequence. So the number sequence $\{S(x_n,x_n,z)\}_{n\in K}$ is convergent. Therefore the number sequence  $\{S(x_n,x_n,z)\}_{n\in \mathbb{N}}$ is $\mathcal{I^*}$-convergent. This gives the number sequence $\{S(x_n,x_n,z)\}_{n\in \mathbb{N}}$ is $\mathcal{I}$-convergent. Therefore the sequence $\{x_n\}_{n\in\mathbb{N}}$ is $\mathcal{I}$-localized and $z\in L$. This proves the theorem.
\end{proof}

\begin{defi}(cf.\cite{Nabiev Savas Gurdal})
Let $(X,S)$ be a $S$-metric space and $\xi\in X$. Then 
 $\xi$ is said to be  $\mathcal{I}$-limit point of the sequence $\{x_n\}_{n\in\mathbb{N}} \in X$ if there is a set $M=\{m_1 <m_2<\cdots \}$ such that $M\notin \mathcal{I}$ and $\lim _{k\to \infty}S(x_{m_k},x_{m_k},\xi)=0$. 
And the point $\xi$ is said to be  $\mathcal{I}$-cluster point of the sequence $\{x_n\}_{n\in\mathbb{N}} \in X$ if and only if for each $\varepsilon >0$ we have $\{n\in \mathbb{N}: S(x_n,x_n,\xi)<\varepsilon\}\notin \mathcal{I}$. 
\end{defi}

\begin{defi}(cf.\cite{Nabiev Savas Gurdal})
Let $\{x_n\}_{n\in\mathbb{N}}$ be a sequence in $X$ and $M=\{m_1 <m_2<\cdots \}\subset \mathbb{N}$. 
If $M\in \mathcal{I}$, then the subsequence $\{x_n\}_{n\in M}$ of  $\{x_n\}_{n\in\mathbb{N}}$ is called $\mathcal{I}$-thin subsequence of  $\{x_n\}_{n\in\mathbb{N}}$. On the other hand, if $M\notin \mathcal{I}$, then the subsequence $\{x_n\}_{n\in M}$ of  $\{x_n\}_{n\in\mathbb{N}}$ is called $\mathcal{I}$-nonthin subsequence of  $\{x_n\}_{n\in\mathbb{N}}$.
\end{defi}

\begin{prop}\label{prop3.2}
If $z\in X$ is an $\mathcal{I}$-limit point (respectively  $\mathcal{I}$-cluster point) of a sequence $\{x_n\}_{n\in \mathbb{N}}\in X$, then for each $y\in X$ the number $S(z,z,y)$ is an $\mathcal{I}$-limit point (respectively  $\mathcal{I}$-cluster point) of the number sequence $\{S(x_n,x_n,y)\}_{n\in \mathbb{N}}$. 
\end{prop}
\begin{proof}
Let $z\in X$ be an $\mathcal{I}$-limit point of $\{x_n\}_{n\in \mathbb{N}}\in X$. Then there is a set $M=\{m_1<m_2<\cdots<m_k<\cdots\}\notin \mathcal{I}$ such that $\lim _{k\to \infty}S(x_{m_k},x_{m_k},z)=0$. Then for each $\varepsilon>0$ there exists $n_0\in\mathbb{N}$ such that $S(x_{m_k},x_{m_k},z)<\frac{\varepsilon}{2}$ for all $k>n_0$. Let $y\in X$. Now by Lemma \ref{lem0.1}, we have $|S(x_{m_k},x_{m_k},y)-S(z,z,y)|\leq 2S(x_{m_k},x_{m_k},z)<\varepsilon$, $\forall k>n_0$. Therefore $\lim _{k\to \infty}S(x_{m_k},x_{m_k},y)=S(z,z,y)$. Hence, according to the definition of $\mathcal{I}$-limit point of a real sequence, we get  $S(z,z,y)$ is an $\mathcal{I}$-limit point of the number sequence $\{S(x_n,x_n,y)\}_{n\in\mathcal{I}}$.

Let $z\in X$ be an $\mathcal{I}$-cluster point of $\{x_n\}_{n\in \mathbb{N}}\in X$. Then for each $\varepsilon>0$ we have $\{n\in\mathbb{N}: S(x_n,x_n,z)<\frac{\varepsilon}{2}\}\notin \mathcal{I}$. Let $y\in X$. Now using the Lemma \ref{lem0.1}, we get $|S(x_n,x_n,y)-S(z,z,y)|\leq 2S(x_n,x_n,z)$. Therefore $\{n\in\mathbb{N}: S(x_n,x_n,z)<\frac{\varepsilon}{2}\}\subset \{n\in\mathbb{N}:|S(x_n,x_n,y)-S(z,z,y)|<\varepsilon\}$. Hence $ \{n\in\mathbb{N}:|S(x_n,x_n,y)-S(z,z,y)|<\varepsilon\}\notin\mathcal{I}$. Therefore the number $S(z,z,y)$ is an $\mathcal{I}$-cluster point of the number sequence $\{S(x_n,x_n,y)\}_{n\in\mathbb{N}}$.
\end{proof}
\begin{thm}(cf.\cite{Tripathy})\label{thm3.4}
Let $x=\{x_n\}_{n\in\mathbb{N}}$ be a real sequence such that $\mathcal{I}\text{-}\lim x_n=\xi$. If $\Lambda_x(\mathcal{I})$ and $\Gamma_x(\mathcal{I})$ are the sets of all $\mathcal{I}$-limit points and $\mathcal{I}$-cluster points of $x$ respectively, then we have $\Lambda_x(\mathcal{I})=\Gamma_x(\mathcal{I})=\{\xi\}$.
\end{thm}

\begin{lem}\label{lemma3.3}
If $\alpha, \beta \in X$ are $\mathcal{I}$-limit points (respectively $\mathcal{I}$-cluster points) of an $\mathcal{I}$-localized sequence $\{x_n\}_{n\in\mathbb{N}}$ then $S(\alpha,\alpha,x)=S(\beta,\beta,x)$ for each $x\in loc_\mathcal{I}(x_n)$.
\end{lem}

\begin{proof}
Let $x\in loc_\mathcal{I}(x_n)$ and $y=\{y_n\}=\{S(x_n,x_n,x)\}_{n\in\mathbb{N}}$. Let $\alpha,\beta$ be any two $\mathcal{I}$-limit points (respectively $\mathcal{I}$-cluster points) of $\{x_n\}_{n\in\mathbb{N}}$. Then by the Proposition \ref{prop3.2}, $S(\alpha,\alpha,x)$, $S(\beta,\beta,x)$ are the $\mathcal{I}$-limit points (respectively $\mathcal{I}$-cluster points) of the number sequence $\{S(x_n,x_n,x)\}_{n\in\mathbb{N}}$ i.e., $S(\alpha,\alpha,x),\ S(\beta,\beta,x)\in \Lambda_y(\mathcal{I})$ (respectively $\Gamma_y(\mathcal{I})$). Since $\{x_n\}_{n\in\mathbb{N}}$ is an $\mathcal{I}$-localized sequence  and $x\in loc_\mathcal{I}(x_n)$, the number sequence $\{S(x_n,x_n,x)\}_{n\in\mathbb{N}}$ is $\mathcal{I}$-convergent. Let $y_n\xrightarrow{\mathcal{I}}\xi$. Then by the Theorem \ref{thm3.4}, $\Lambda_y(\mathcal{I})=\Gamma_y(\mathcal{I})=\{\xi\}$. Therefore $S(\alpha,\alpha,x)=S(\beta,\beta,x)$ for each $x\in loc_\mathcal{I}(x_n)$. This completes the proof.
\end{proof}

\begin{lem}
$loc_{\mathcal{I}}(x_n)$ does not contain more than one $\mathcal{I}$-limit point (respectively $ \mathcal{I}$-cluster point)  of the sequence $\{x_n\}_{n\in \mathbb{N}}$ in $X$.
\end{lem}

\begin{proof}
If possible, let $z_1,z_2 \in loc_{\mathcal{I}}(x_n)$ be two distinct $\mathcal{I}$-limit points (respectively $\mathcal{I}$-cluster points) of the sequence $\{x_n\}_{n\in \mathbb{N}}$. then by the  Lemma \ref{lemma3.3}, we have $S(z_1,z_1,z_1)=S(z_2,z_2,z_1)$. But $S(z_1,z_1,z_1)=0$. Consequently $S(z_2,z_2,z_1)=0$. This gives $z_1=z_2$ which leads to a contradiction. This proves the lemma.
\end{proof}
\begin{rem}
We know from the Theorem \ref{thm3.4} that if $\{x_n\}_{n\in\mathcal{I}}$ is $\mathcal{I}$-convergent to $x$ then $\mathcal{I}$-limit point is unique. But converse result holds if the $\mathcal{I}$-limit point belongs to $\mathcal{I}$-locator of $\{x_n\}_{n\in\mathbb{N}}$ which is shown in the following proposition.
\end{rem}

\begin{prop}
If the sequence $\{x_n\}_{n\in \mathbb{N}}$ has an $\mathcal{I}$-limit point $y\in  loc_{\mathcal{I}}(x_n)$, then $\mathcal{I}-\lim _{n\to \infty}x_n =y$.
\end{prop}

\begin{proof}
Since $y\in loc_\mathcal{I}(x_n)$ is an $\mathcal{I}$-limit point of $\{x_n\}_{n\in\mathbb{N}}$, then by the Proposition \ref{prop3.2}, $S(y,y,y)$ is an $\mathcal{I}$-limit point of the number sequence $\{S(x_n,x_n,y)\}_{n\in\mathbb{N}}$.
By the condition $y\in loc_\mathcal{I}(x_n)$, the number sequence $t=\{t_n\}=\{S(x_n,x_n,y)\}_{n\in\mathbb{N}}$ is $\mathcal{I}$-convergent. Let $\mathcal{I}\text{-}\lim_{n\to\infty}S(x_n,x_n,y)=\xi$. Now since $S(y,y,y)\in\Lambda_t(\mathcal{I})$ and by the Theorem \ref{thm3.4} we have $\Lambda_t(\mathcal{I})=\{\xi\}$, therefore $S(y,y,y)=\xi$. So $\mathcal{I}\text{-}\lim_{n\to\infty}S(x_n,x_n,y)=\xi=S(y,y,y)=0$ i.e., $\mathcal{I}\text{-}\lim_{n\to\infty}S(x_n,x_n,y)=0$. So for each $\varepsilon >0$ the set $\{n\in\mathbb{N}: S(x_n,x_n,y)\geq \varepsilon\}\in\mathcal{I}$ which gives $\mathcal{I}\text{-}\lim_{n\to\infty} x_n=y$. This completes the proof.
\end{proof}

\begin{defi}(cf.\cite{Nabiev Savas Gurdal})
Let  $\{x_n\}_{n\in \mathbb{N}}$ be $\mathcal{I}$-localized sequence with the $\mathcal{I}$-locator $L=loc_{\mathcal{I}}(x_n)$. Then the number $\sigma= \text{inf}_{x\in L}(\mathcal{I}-\lim _{n\to \infty} S(x_n,x_n,x))$ is called the $\mathcal{I}$-barrier of $\{x_n\}_{n\in \mathbb{N}}$.
\end{defi}

\begin{thm}\label{thm3.5}
Let $\mathcal{I}$ satisfies the condition $(AP)$. Then an $\mathcal{I}$-localized sequence is an $\mathcal{I}$-Cauchy sequence if and only if $\sigma=0$.
\end{thm}

\begin{proof}
Let $\{x_n\}_{n\in \mathbb{N}}$ be an $\mathcal{I}$-Cauchy sequence in $X$. So it is $\mathcal{I^*}$-Cauchy sequence, since $\mathcal{I}$ satisfies the condition $(AP)$. Therefore there exists a set $K=(k_n)$ such that $K\in \mathcal{F(I)}$ and $\lim _{n,m\to \infty}S(x_{k_n},x_{k_n},x_{k_m})=0$. So for each $\varepsilon >0$, there exists $n_0 \in \mathbb{N}$ such that $S(x_{k_n},x_{k_n},x_{k_{n_0}})<\varepsilon$ for all $n\geq n_0$. Since  $\{x_n\}_{n\in \mathbb{N}}$ is $\mathcal{I}$-localized sequence,  $\mathcal{I}-\lim_{n\to \infty}S(x_n,x_n,x_{k_{n_0}})$ exists. Therefore we have $\mathcal{I}-\lim_{n\to \infty}S(x_{k_n},x_{k_n},x_{k_{n_0}})\leq \varepsilon$. Hence $\sigma \leq \varepsilon$. As, $\varepsilon >0$, $\sigma=0$. 

Conversely assume that $\sigma=0$. Then by definition of $\sigma$, for each $\varepsilon >0$ there is an $x\in loc_{\mathcal{I}}(x_n) $ such that $\beta (x)=\mathcal{I}-S(x_n,x_n,x)<\varepsilon$. So $\{n\in \mathbb{N}: |S(x_n,x_n,x)-\beta(x)|\geq \varepsilon -\beta(x)\}\in \mathcal{I}$, as $\varepsilon -\beta(x) >0$. Now infact, since $S(x_n,x_n,x)=|S(x_n,x_n,x)-\beta(x)+\beta(x)|\leq |S(x_n,x_n,x)-\beta(x)|+\beta(x)$, therefore $\{n\in \mathbb{N}: S(x_n,x_n,x)\geq \varepsilon \}\in \mathcal{I}$ i.e. the sequence $\{x_n\}_{n\in\mathbb{N}}$ is $\mathcal{I}$-convergent. Consequently $\{x_n\}_{n\in \mathbb{N}}$ is an $\mathcal{I}$-Cauchy sequence. This proves the theorem.
\end{proof}

\begin{rem}
From the proof of the above theorem we can conclude that converse part holds without the condition $(AP)$.
\end{rem}

\begin{thm}
If the sequence $\{x_n\}_{n\in \mathbb{N}}$  is $\mathcal{I}$-localized in itself and $\{x_n\}_{n\in \mathbb{N}}$ contains an $\mathcal{I}$-nonthin Cauchy subsequence, then $\{x_n\}_{n\in \mathbb{N}}$ is an $\mathcal{I}$-Cauchy sequence.
\end{thm}

\begin{proof}
Let $\{y_n\}_{n\in \mathbb{N}}$ be an $\mathcal{I}$-nonthin Cauchy subsequence of $\{x_n\}_{n\in \mathbb{N}}$. Without any loss of generality we suppose that all the members of  $\{y_n\}_{n\in \mathbb{N}}$ are in $loc_{\mathcal{I}}(x_n) $. Since $\{y_n\}_{n\in \mathbb{N}}$ is a Cauchy sequence, then by the Theorem \ref{thm3.5},  we have $inf_{y_n\in loc_{\mathcal{I}}(x_n)} \mathcal{I}\text{-}\lim _{m\to \infty}S(y_m,y_m,y_n)=0$. Now Since  $\{x_n\}_{n\in \mathbb{N}}$  is $\mathcal{I}$-localized in itself then the number sequence $\{S(x_m,x_m,y_n)\}_{m\in \mathbb{N}}$, $y_n\in loc_{\mathcal{I}}(x_n)$, is $\mathcal{I}$-convergent. Therefore we have  $\mathcal{I}\text{-}\lim _{m\to \infty}S(x_m,x_m,y_n)=\mathcal{I}\text{-}\lim _{m\to \infty}S(y_m,y_m,y_n)=0$. This shows that $\sigma=0$. Therefore by the Theorem \ref{thm3.5} we have $\{x_n\}_{n\in \mathbb{N}}$ is an $\mathcal{I}$-Cauchy sequence. This completes the proof.
\end{proof}

\textbf{Acknowledgement.}
The second author is grateful to The Council of Scientific and Industrial Research (CSIR), HRDG, India, for the grant of Junior Research Fellowship during the preparation of this paper.

\end{document}